%% file: surcom.tex
\documentclass[11pt, a4paper]{amsart}
\usepackage[utf8]{inputenc}
\usepackage{amsmath,amsthm,comment}
\usepackage{amsfonts}
\usepackage{amssymb}

\usepackage{graphicx,xy,pstricks}
\input xy
\xyoption{all}
\newcommand{\executeiffilenewer}[3]{%
 \ifnum\pdfstrcmp{\pdffilemoddate{#1}}%
 {\pdffilemoddate{#2}}>0%
 {\immediate\write18{#3}}\fi%
}

\theoremstyle{definition} 

 \newtheorem{definition}{Definition}[section]
 \newtheorem{remark}[definition]{Remark}
 \newtheorem{example}[definition]{Example}

\theoremstyle{plain}      

 \newtheorem{proposition}[definition]{Proposition}
 \newtheorem{theorem}[definition]{Theorem}
 
 \newtheorem{lemma}[definition]{Lemma}

\newtheorem*{theorem*}{Theorem}

\newcommand{\R}{\mathbb{R}}
\newcommand{\C}{\mathbb{C}}
\newcommand{\Q}{\mathbb{Q}}
\newcommand{\N}{\mathbb{N}}

\newcommand{\Z}{\mathbb{Z}}

\DeclareMathOperator{\mcg}{Mod}

\DeclareMathOperator{\St}{St}

\DeclareMathOperator{\SL}{SL}

\DeclareMathOperator{\aff}{Aff}

\DeclareMathOperator{\id}{Id}

\DeclareMathOperator{\cl}{cl}

\DeclareMathOperator{\ocl}{ocl}
\DeclareMathOperator{\socl}{socl}
\DeclareMathOperator{\aut}{Aut}
\DeclareMathOperator{\ev}{ev}

\title[Overcommuting pairs in groups and 3-manifolds bounding them]{Overcommuting pairs in groups and \\ 3-manifolds bounding them}
\begin{document}
\date{}
\author{Livio Liechti} 
\address{Department of Mathematics, University of Fribourg, Ch. du Musée 23, 1700 Fribourg, Switzerland}
\email{livio.liechti@unifr.ch}

\author{Julien March\'e}
\address{Sorbonne Universit\'e, IMJ-PRG, 75252 Paris c\'edex 05, France}
\email{julien.marche@imj-prg.fr}

\begin{abstract}
We introduce the notions of overcommutation and overcommutation length in groups,
and show that these concepts are closely related to representations of the fundamental groups of 3-manifold and their Heegaard genus. 
We give many examples including translations in the affine group of the line and provide upper bounds for the 
overcommutation length in~$\SL_2$, related to the Steinberg relation. 
\end{abstract}

\maketitle
\section{Introduction}
We say that two commuting elements $g,h$ in a group $G$ overcommute if their lifts $\tilde{g},\tilde{h}$ in any central extension $\tilde{G}$ of $G$ still commute. It is equivalent to ask that some class associated to $(g,h)$ in $H_2(G,\Z)$ vanishes (see Section~\ref{overcommute}). By bordism arguments, it is also equivalent to ask that there exists a connected and oriented 3-manifold $M$ with torus boundary and a morphism $\rho:\pi_1(M)\to G$ mapping the generators of the fundamental group of the boundary to $g$ and $h$ respectively. We will say that $M$ is an overcommuting manifold for $(g,h)$. 

This latter point of view was explained to us by Ghys some years ago and proves the existence of 3-manifolds with specific properties, with the following paradigmatic example. 

Let $k$ be a field containing $\frac{1}{6}$. Results of Steinberg of the sixties imply that for any $x\in k\setminus\{0,1\}$, the matrices 
$$\begin{pmatrix} x & 0 \\ 0 & x^{-1} \end{pmatrix}\text{ and }\begin{pmatrix} (1-x) & 0 \\ 0 & (1-x)^{-1} \end{pmatrix}$$
overcommute in $\SL_2(k)$, see \cite{Hutchinson} Section 3 for a condensed exposition or Section~\ref{effectif} of this article. The reader familiar with $K$-theory will relate this fact to the Steinberg relation\footnote{More precisely $2\{x,1-x\}=0$} $\{x,1-x\}=0$. Taking for instance $k=\C(x)$, one gets a 3-manifold $M$ which may be seen as a topological counterpart of the Steinberg relation. Although we will not need it in this article, its defining property can be formulated in terms of the $A$-polynomial of $M$ by saying that it is divisible by $\mathcal{L}+\mathcal{M}-1$ in the notation of \cite{ccgls}. Computer experiments show that the complexity of the $A$-polynomials tends to grow very quickly and we could not find an $A$-polynomial with this property in the census of the 200 simplest 3-manifolds. In this article we will show how such a 3-manifold can be effectively constructed and prove that its Heegaard genus is less than 26. As far as we know, the motivation of Ghys was to provide a kind of ``topological proof" of the Steinberg relation, observing that all proofs of this fundamental relation in K-theory involve obscure computations. Clearly, we do not fulfil his expectations as we construct an obscure 3-manifold from a well-known and obscure proof of the Steinberg relation. In the last section, we provide extra motivations for finding a nice Ghys manifold.

Going back to the original problem, we can use the Hopf formula to reformulate the overcommutation in terms of a presentation $G=F/R$ where~$F$ is a free group. Two commuting elements $g,h$ in $G$ overcommute if and only if there are lifts $\tilde{g},\tilde{h}$ in~$F$ such that $[\tilde{g},\tilde{h}]\in [F,R]$. We define the overcommutation length of the pair~$(g,h)$ as the minimal number of commutators in~$[F,R]$ needed to write down such an expression and we denote it by~$\ocl(g,h)$. 
Surprisingly, this number does not depend on the presentation and may be interpreted as a complexity of the overcommuting pair~$(g,h)$ reminiscent of the commutator length, see for instance~\cite{scl}.
Our first result is the following: 

\begin{theorem}\label{T1}
Given two elements $g,h\in G$ that overcommute, $\ocl(g,h)+1$ is the minimal Heegaard genus of an overcommuting manifold for $(g,h)$.
\end{theorem}
Moreover, the proof is constructive in the sense that one can algorithmically produce a 3-manifold from an expression of $[\tilde{g},\tilde{h}]$ in $[F,R]$ and vice-versa. 

We then study in detail the case of the affine group of transformations of the form $z\mapsto az+b$ with $a\in k^*$ and $b\in k$. One can show that translations overcommute and one can even find a manifold which is overcommuting for all pairs of translations. However, these manifolds are not so easy to find. They have the following nice interpretation: 

\begin{proposition}\label{oclaffine}
A 3-manifold is overcommuting for translations if and only if there exists a morphism $\lambda:\pi_1(M)\to k^*$ (the linear part) mapping the boundary to 1 and such that the natural map $H_1(\partial M,k_\lambda)\to H_1(M,k_\lambda)$ vanishes. 

The minimal Heegaard genus of such manifolds is 3, and it is achieved by the complement of a knot in the 0-surgery over the stevedore knot $6_1$. 
\end{proposition}
Here, the notation $k_\lambda$ means the vector space $k$ with $\gamma\in \pi_1(M)$ acting by $\gamma.x=\lambda(\gamma)x$.  Notice that $\lambda$ has to be non trivial otherwise the statement would contradict Poincaré duality. In particular, $M$ cannot be a knot complement in~$S^3$. 

Finally, we study the case of $\SL_2(k)$. Our first task is to replace this group by its universal central extension $\St_2(K)$ where commutation is equivalent to overcommutation. The latter group was introduced by Steinberg in terms of a presentation $F/R$:  we denote by $S\subset F$ the set of relations defined by Steinberg which normally generate $R$. 
For $x\in R$, we also denote by $l_S(x)$ the minimal number of conjugates of elements of $S\cup S^{-1}$ needed to write~$x$. Contrary to $\ocl$, this number strongly depends on the presentation, but it is much easier to compute.

Our main result is then the following:

\begin{theorem}\label{T2} Let $k$ be a field containing $\frac{1}{6}$ and $\sqrt{2}$ and let $\St_2(k)=F/R$ be the standard presentation of the Steinberg group. For any $g,h\in F$ which commute in $\St_2(k)$, we have
$$\ocl(g,h)\le 5 l_S([g,h])+2.$$
\end{theorem}
Said informally, the overcommutation length $\ocl(g,h)$ is controlled by the number of relations needed to prove that $g$ and $h$ commute in $\St_2(k)$. This contraction property looks non trivial and is shared by most $1$-relator groups (see for instance Example \ref{one-relator}).

To end this introduction, we observe that one can define a simplicial volume $||[g,h]||$ of an overcommuting pair in the following way. Let $BG$ be a classifying space for $G$ and let $f:S^1\times S^1\to BG$ be a continuous map such that $f_*$ maps the generators of $\pi_1(S^1\times S^1)$ to $g$ and $h$ respectively. For $\epsilon>0$, consider the set $X_\epsilon$ of singular 2-cycles $x=\sum_i \lambda_i \sigma_i\in Z_2(S^1\times S^1,\R)$ representing the fundamental class $[S^1\times S^1]$ and such that $\sum_i |\lambda_i|\le \epsilon$. We define 
$$||f||_\epsilon=\inf\Big\{\sum_j |\mu_j|, y=\sum \mu_j\sigma_j\in C_3(BG,\R), \partial y\in f_*X_\epsilon\Big\}$$
and set $||[g,h]||=\liminf_{\epsilon\to 0} ||f||_\epsilon$. This definition is directly inspired from \cite{thurston}, Chap.~6. It is easy to show that this is well-defined and that if $G=\pi_1(M)$ and~$g,h$ are the generators of the fundamental group of the boundary, then $||[g,h]||$ coincides with the simplicial volume of $M$.  

Moreover, the following immediate proposition shows that the simplicial volume is a lower bound for the complexity of any overcommuting manifold.
\begin{proposition}\label{volume}
Let $g,h\in G$ be an overcommuting pair. For any overcommuting manifold $M$ for $(g,h)$ we have $$||[g,h]||\le ||M||.$$
\end{proposition}

For the example of two translations in the affine group, we prove the following proposition in Section~\ref{OCMT}.

\begin{proposition}\label{vol_OCMT}
Let~$k$ be a field containing~$\frac{1}{6}$. If $t_1$ and $t_2$ are two translations in the affine group $\aff(k)$, then $||[t_1,t_2]||=0$. 
\end{proposition}
However, we have more questions than answers concerning this simplicial volume.
Natural questions include: is the simplicial volume of Ghys' example positive? Or does there always exists a 3-manifold where the inequality of Proposition~\ref{volume} is an equality? Having little to say about this for the moment, this topic will not be expanded further in this article.
 \medskip

{\bf Plan of the paper.} In Section 2, we define carefully the notion of overcommutation and give many examples. In Section 3, we define the overcommutation length and prove Theorem \ref{T1}. Section 4 is devoted to the case of the affine group and Section 5 to $\SL_2(k)$, where we prove Theorem~\ref{T2}.  
 
 {\bf Acknowledgements.} We would like to thank \'E. Ghys for asking and re-asking about the mysterious Steinberg manifold and J. Barge for his interesting remarks. 
We also thank an anonymous referee for helpful comments and suggestions. The first author was partially supported by the Swiss National Science Foundation (grant nr.\ 175260)
 
\section{Overcommuting pairs}\label{overcommute}
In a group~$G$, we write~$[g,h]=ghg^{-1}h^{-1}$ for~$g,h\in G$. We fix a classifying space~$BG$ and recall that for any connected CW-complex~$X$, homotopy classes of maps~$f:X\to BG$ are in bijection with morphisms~$f_*:\pi_1(X)\to G$. 

Hence, topologically, a group element~$g\in G$ lies in the commutator subgroup~$[G,G]$ exactly if there exists a compact orientable surface~$S$ with one boundary component and a continuous map from~$S$ to~$BG$ so that the boundary~$\partial S$ is mapped to the loop corresponding to~$g$. 

We provide an analogous criterion for the torus in~$BG$ defined by a pair~$(g,h)$ of commuting elements~$g,h\in G$ to bound a compact orientable~3-manifold~$M$ with toric boundary. Precisely, the map~$\phi_*:\Z^2\to G$ given by~$\phi_*(m,n)=g^mh^n$ corresponds to a continuous map~$\phi:S^1\times S^1\to BG$. 
Bounding means there exists a continuous map~$\Phi:M\to BG$ which extends~$\phi:\partial M=S^1\times S^1\to BG$.  
The following proposition is a combination of well-known arguments. 

\begin{proposition}\label{equivalence}
Let $G$ be a group and $g,h$ be two commuting elements of~$G$. The following assertions are equivalent. 
\begin{enumerate}
\item[(i)] In any central extension~$1\to Z\to \tilde{G}\to G\to 1$, any lifts~$\tilde{g},\tilde{h}\in\tilde{G}$ of~$g$ and~$h$ respectively commute.
\item[(ii)] Define the morphism~$\phi_*:\Z^2\to G$ by~$\phi(m,n)=g^mh^n$. Then the map $$H_2(\Z^2,\Z)\to H_2(G,\Z)$$ induced by $\phi_*$ vanishes.
\item[(iii)] Given an extension $1\to R\to F\to G\to 1$ where $F$ is a free group and lifts $\tilde{g},\tilde{h}\in F$ of $g$ and $h$ respectively, we have $[\tilde{g},\tilde{h}]\in [F,R]$. 
\item[(iv)] There exists a compact orientable 3-manifold $M$ with $\partial M=S^1\times S^1$ and a representation $\rho:\pi_1(M)\to G$ such that 
$$g=\rho(m)\text{ and }h=\rho(l)$$
 where $m$ and $l$ are the homotopy classes of $S^1\times\{1\}$ and $\{1\}\times S^1$ respectively.
\end{enumerate}
\end{proposition}

Observe that in the Properties~(i) and~(iii), the condition does not depend on the chosen lifts. 

\begin{definition}
Let~$G$ be a group. If $g,h\in G$ satisfy the equivalent properties in Proposition~\ref{equivalence}, we say that they overcommute. 
A manifold $M$ satisfying Property (iv) of Proposition~\ref{equivalence} is called an overcommuting manifold for the pair~$(g,h)$. 
\end{definition}

\begin{proof}[Proof of Proposition~\ref{equivalence}]
The equivalence of~(ii) and~(iii) is the content of the following Hopf formula:~$H_2(G,\Z)=R\cap [F,F]/[F,R]$ and~$\phi_*([S^1\times S^1])=[\tilde{g},\tilde{h}]$, where~$\tilde{g}$ and~$\tilde{h}$ are any lifts of~$g$ and~$h$ in~$F$. 

To prove~(i)$\implies$(iii), observe that the sequence~$$1\to R/[F,R]\to F/[F,R]\to G\to 1$$ is a central extension. Hence given lifts~$\tilde{g},\tilde{h}$ of~$g$ and~$h$ in~$F$, we must have by property~(i) that~$[\tilde{g},\tilde{h}]$ vanishes in~$F/[F,R]$, hence the result.

Reciprocally, we observe that as~$F$ is free, there is a morphism~$\Phi:F\to \tilde G$ making the following diagram commutative:
$$\xymatrix{ 1\ar[r] & R \ar[r]\ar[d] & F\ar[r]\ar[d]^\Phi & G\ar[r]\ar[d]^{\id}& 1 \\ 1\ar[r] & Z \ar[r] & \tilde{G}\ar[r] & G\ar[r]& 1}$$
Take~$\tilde{g},\tilde{h}$ lifts of~$g,h$ in~$F$ so that~$\Phi(\tilde{g}),\Phi(\tilde{h})$ are lifts of~$g,h$ in~$\tilde{G}$. By property~(iii), one can write~$[\tilde{g},\tilde{h}]=\prod_{i=1}^k [f_i,r_i]^{n_i}$ for~$f_i\in F,r_i\in R$ and~$n_i\in \Z$. We get$[\Phi(\tilde{g}),\Phi(\tilde{h})]=\prod_{i=1}^k[\Phi(f_i),\Phi(r_i)]^{n_i}$. This vanishes because~$\Phi(r_i)\in Z$, which by assumption lies in the center of~$\tilde{G}$. 
 
To prove~(iv)$\implies$(ii), we recall that $H_2(\Z^2,\Z)=H_2(S^1\times S^1,\Z)$ is generated by the fundamental class $[S^1\times S^1]$ and that the representation~$\rho:\pi_1(M)\to G$ is induced by a map~$\Phi:M\to BG$ so that we have~$\Phi_*([S^1\times S^1])=\Phi_*([\partial M])$. As we can write~$[\partial M]=\partial z$ where~$z\in C_3(M)$ represents the fundamental class of~$M$ relative to the boundary, we get $\Phi_*([S^1\times S^1]=\partial \Phi_*z=0\in H_2(BG,\Z)$. 

Reciprocally, we can use bordism groups and observe~$\Omega_2(BG)=H_2(BG,\Z)=H_2(G,\Z)$. Hence the vanishing of~$\Phi_*([S^1\times S^1])$ implies the existence of a 3-manifold~$M$ with~$\partial M=S^1\times S^1$ and an extension~$\Phi:M\to BG$ of the boundary map~$\Phi:S^1\times S^1\to BG$. We will give below an alternative and more constructive proof in Theorem~\ref{construction}.
\end{proof}

By Proposition~\ref{equivalence}, examples of overcommuting pairs are given by the elements~$m,l\in \pi_1(M)$, where~$M$ is a compact oriented 3-manifold with toric boundary and~$m,l$ are generators of~$\pi_1(\partial M)$. A reformulation of the proposition states that these examples are universal in the sense that any other example is the homomorphic image of a topological one. One can also restrict to irreducible ones as any 3-manifold $M$ with torus boundary can be written $M=M'\#M''$ where $M'$ is closed and $M''$ is irreducible with torus boundary.

\begin{remark}
The group~$\SL_2(\Z)$ acts on overcommuting pairs by monomial transformations generated by~$(g,h)\mapsto (g,gh)$ and~$(g,h)\mapsto (gh,h)$. At the level of the overcommuting 3-manifold, it simply consists in reparametrizing the boundary torus. In the sequel, we will freely use this action.
\end{remark}

We end this section with some examples and constructions of overcommuting pairs~$(g,h)$. 

\begin{example}\label{surfacecroixcercle}
Suppose that~$g,h$ are two elements in a group~$G$ such that we have~$h=\prod_{i=1}^{n}[x_i,y_i]$, where~$x_1,\ldots,y_n$ commute with~$g$. Then~$g$ and~$h$ overcommute for the following topological reason: let~$\Sigma$ be a surface with genus~$n$ and 1 boundary component. One can define a morphism~$\pi_1(\Sigma\times S^1)\to G$ by sending the class of~$S^1$ to~$g$ and the standard generators of~$\pi_1(\Sigma)$ to~$x_1,y_1,\ldots,x_n,y_n$. The manifold~$\Sigma\times S^1$ has a toric boundary and is an overcommuting manifold for the pair~$(g,h)$. 
An explicit example is given by a pair of disjoint and non-separating Dehn twists on a surface of genus $g>3$. 
\end{example}

\begin{example}
\label{product_ex}
Suppose that~$g,h_1,h_2$ are three elements of~$G$ such that~$(g,h_1)$ and~$(g,h_2)$ are overcommuting pairs. Then, in any central extension~$\tilde G$ of~$G$ one has~$\tilde g \tilde h_1=\tilde h_1 \tilde g$ and~$\tilde g \tilde h_2=\tilde h_2 \tilde g$. In particular,~$[\tilde{g},\tilde{h}_1\tilde{h}_2]=1$ and~$(g,h_1h_2)$ is an overcommuting pair by Property~(i) of Proposition~\ref{equivalence}. This can be obtained topologically from two overcommuting manifolds~$M_1$ and~$M_2$ for~$(g,h_1)$ and~$(g,h_2)$, respectively, by gluing them along the annulus embedded in their boundary and mapping to $g$.   
In the case of knot complements, this operation is equivalent to the connected sum. 
\end{example}

\begin{example}
Let~$g,h$ be two elements of~$G$ and set~$g_n=h^n g h^{-n}$. If~$g$ and~$g_1$ commute, then~$g$ and~$g_1g_{-1}$ overcommute as the following proof shows. We choose a central extension~$\tilde G$ of~$G$ and lifts~$\tilde g$ and~$\tilde h$ of~$g$ and~$h$. We set~$\tilde{g}_n=\tilde h^n \tilde g\tilde h^{-n}$. 
By assumption, there exists~$z$ in the center of~$\tilde G$ such that~$\tilde g_1\tilde g=z\tilde g\tilde g_1$. Conjugating this equation by~$\tilde h^{-1}$, we get~$\tilde g\tilde g_{-1}=z\tilde g_{-1}\tilde g$ and the result follows. 
The topological counterpart of this computation is that the group~$\langle g,h | [g,hgh^{-1}]=1\rangle$ is the fundamental group of a 3-manifold with toric boundary, precisely the~0-surgery on one component of the Whitehead link (which is Seifert fibered). We observe also that as its abelianization is~$\Z^2$, this is not a knot complement in~$S^3$. 
\end{example}

\begin{example}
Let $c\in G$ be a central element. It is easy to show that the map $G\to H_2(G,\Z)$ mapping $g$ to the class of the commuting pair $(g,c)$ is a morphism  and hence defines a map $H_1(G,\Z)\to H_2(G,\Z)$. Elements in the kernel of this map give interesting overcommuting pairs. For instance if $G=\SL_2(\Z)$ we have $H_1(G,\Z)=\Z/12\Z$ and $H_2(G,\Z)=0$. Topologically, the fundamental group of the trefoil knot surjects to $\SL_2(\Z)$ and maps the longitude to the central element $-\id$. 
\end{example}
\begin{example}
Let $F_2$ be the group freely generated by two elements $u$ and~$v$ and let $w$ be an element of $F_2$. We consider the group $G_w=\langle u,v|r\rangle$ where $r=wuw^{-1}v^{-1}$. We also define $\phi\in \aut(F_2)$ by $\phi(u)=u^{-1}$ and $\phi(v)=v^{-1}$. 
\begin{definition}
Let us call the group $G_w$ a two-bridge group if there exists $g\in G$ such that $\phi(r)=gr^{-1}g^{-1}$. 
\end{definition}
As the notation suggests, the fundamental group of a two-bridge knot complement is a two-bridge group. 
In general, we observe that two-bridge groups have the following properties: 
\begin{enumerate}
\item $H_1(G_w,\Z)=\Z.$
\item $G_w$ is normally generated by $u$ (or $v$). 
\item $H_2(G_w,\Z)=0.$
\item The map $\phi$ induces an automorphism of $G_w$.
\item The elements $u$ and $l=\phi(w)^{-1}w$ commute (hence overcommute).
\end{enumerate}
It follows that there exists an overcommuting manifold~$M$ for~$(l,u)$, i.e.~a morphism~$\rho:\pi_1(M)\to G_w$. It looks interesting to understand better this map~$\rho$. For instance, does it define epimorphisms between 2-bridge knot groups as in \cite{ors}?
\end{example}

\section{Overcommutation length}
For an element~$g\in[G,G]$, the commutator length~$\cl(g)$ is the minimum number of commutators needed to write~$g$ as a product of commutators. Topologically,~$\cl(g)$ is the minimal genus among compact surfaces with one boundary component bounding~$g$ in a classifying space~$BG$. We define an analogue measure of the complexity of an overcommuting pair, the overcommutation length. 

\begin{definition}
\label{ocl_def}
Let~$(g,h)$ be an overcommuting pair of elements~$g,h\in G$, and let~$1\to R\to F\to G\to 1$ be a presentation of~$G$. 
We define the overcommutation length~$\ocl(\tilde g,\tilde h)$ of two lifts~$\tilde g$ and~$\tilde h$ of~$g$ and~$h$, respectively, to be 
$$\ocl(\tilde g,\tilde h)=\min\{k\in \N, [\tilde{g},\tilde{h}]=\prod_{i=1}^k [f_i,r_i]^{\pm 1}\in F\},$$
where $f_1,\ldots,f_k\in F$ and $r_1,\ldots,r_k\in R$.
We also set $\ocl(g,h)=\min \ocl(\tilde g, \tilde h)$ where the minimum is taken over all choices of lifts $\tilde g$ and $\tilde h$ of $g$ and $h$.
\end{definition}
\begin{remark}\label{biendef}We justify this definition by the following series of remarks:
\begin{enumerate}
\item 
By Property~(iii) of Proposition~\ref{equivalence}, the minimum in Definition~\ref{ocl_def} is finite.
\item 
If $\tilde{g}$ and $\tilde{h}$ are lifts of $g$ and $h$, any other lifts have the form $\tilde{g}r$ and $\tilde{h}s$ for $r,s\in R$. We compute 
$$[\tilde g r,\tilde h s]=\tilde g [r,\tilde h] \tilde{g}^{-1}[\tilde g,\tilde h]\tilde h [\tilde g r,s]\tilde h^{-1}=[\tilde g r \tilde g ^{-1},\tilde g\tilde h\tilde g^{-1}][\tilde g,\tilde h][\tilde h\tilde g r\tilde h^{-1},\tilde h s \tilde h ^{-1}]$$
and conclude that $\ocl(\tilde g r,\tilde h s)\le \ocl(\tilde g,\tilde h)+2$. In particular, the overcommutation length does not depend strongly on the lifts. 
\item It does actually depend on the lift as shown by the following example: take $a\in F\setminus R$ and $r\in R\setminus \{1\}$, then $[a,a^2]=1$ and $[ar,a^2]\ne 1$.
\item The overcommutation length does not depend on the presentation as if $G=F'/R'$, there is a morphism $\phi:F\to F'$ mapping $R$ to $R'$. 
Applying~$\phi$ to the identity $[\tilde{g},\tilde{h}]=\prod_{i=1}^k [f_i,r_i]^{\pm 1}$ shows $\ocl(\phi(\tilde g),\phi(\tilde h))\le \ocl(\tilde g,\tilde h)$ and the result follows. 
\end{enumerate}
\end{remark}

In what follows we prove that~$\ocl(g,h)+1$ equals the minimal Heegaard genus among all overcommuting manifolds for the pair~$(g,h)$. Before stating and proving this result, we recall the notion of Heegaard decompositions.

Let~$H_{k+1}$ be a standard handlebody of genus~$k+1\ge1$ in~$\R^3$, and furthermore let~$\Sigma=\partial H_{k+1}$. 
We fix standard generators~$a_1,b_1,\ldots,a_{k+1},b_{k+1}\in \pi_1(\Sigma)$ such that~$b_1,\ldots,b_{k+1}$ bound embedded discs in~$H_{k+1}$ and such that the relation~$[a_1,b_1]\cdots[a_{k+1},b_{k+1}]=1$ holds in~$\pi_1(\Sigma)$. We denote by~$T$ a standard solid torus in~$H_{k+1}$ such that the fundamental group of its boundary is generated by curves homotopic to~$a_1$ and~$b_1$ in~$H_{k+1}$. We will identify~$\partial T$ with the standard torus such that~$a_1$ and~$b_1$ correspond to~$m$ and~$l$, respectively.

\begin{definition} Let~$M$ be a compact oriented~$3$-manifold, with boundary $\partial M=S^1\times S^1$. 
A Heegaard decomposition of~$M$ of genus~$k+1\ge1$ is a homeomorphism~$$M\simeq(H_{k+1}\setminus T)\cup_\phi \overline{H}_{k+1}$$
where~$\phi\in \mcg(\Sigma)$ is an element of the mapping class group of~$\Sigma$ and~$\overline{H}_{k+1}$ denotes a copy of the handlebody~$H_{k+1}$ with opposite orientation.
The Heegaard genus of~$M$ is the minimal genus of a Heegaard decomposition of~$M$. 
\end{definition}

\begin{theorem}\label{construction}
Let~$G$ be a group and let~$(g,h)$ be a pair of overcommuting elements~$g,h\in G$. Then, $\ocl(g,h)+1$ is the minimal Heegaard genus among overcommuting manifolds~$M$ for the pair $(g,h)$. 
\end{theorem}

\begin{proof}
Let~$1\to R\to F\to G\to 1$ be any presentation of the group~$G$, and let~$M$ be an overcommuting manifold for~$g$ and~$h$ with minimal Heegaard genus~$k+1$. One can write~$M=(H_{k+1}\setminus T)\cup_\phi\overline{H}_{k+1}$ and the representation ~$\rho:\pi_1(M)\to G$ satisfies~$\rho(a_1)=g, \rho(b_1)=h,\rho(b_i)=1$ for~$1<i\le k+1$.

Set~$F_{k+1}=\pi_1(\overline{H}_{k+1})$ and observe that the inclusion~$\overline{H}_{k+1}\to M$ induces a surjection~$F_{k+1}\to \pi_1(M)$. As~$F_{k+1}$ is free, one can find a morphism~$\overline{\rho}$ making the following diagram commutative:
$$\xymatrix{ F_{k+1}\ar[r]^{\overline{\rho}}\ar[d]& F\ar[d]^p \\ \pi_1(M) \ar[r]^\rho& G.}$$
Consider now the composition~$\pi_1(\Sigma)\to \pi_1(\overline{H}_{k+1})=F_{k+1}\to F$ and denote the images of the generators by~$\tilde{a}_1,\tilde{b}_1,\ldots,\tilde{a}_{k+1},\tilde{b}_{k+1}$. 
By construction,~$g=p(\tilde{a}_1)$ and~$h=p(\tilde{b}_1)$. Furthermore, $p(\tilde{b}_i)=1$ and hence~$\tilde b_i\in R$ for~$1<i\le k+1$. 
The equality~$\prod_{i=1}^{k+1}[\tilde{a}_i,\tilde{b}_i]=1$ in~$F$ shows that one has the following identity which proves~$\ocl(g,h)\le k$: 
$$[\tilde{a}_1,\tilde{b}_1]=[\tilde{a}_{k+1},\tilde{b}_{k+1}]^{-1}\cdots[\tilde{a}_2,\tilde{b}_2]^{-1}.$$

Suppose now that we have a presentation~$1\to R\to F\to G\to 1$ and a formula~$[\tilde{g},\tilde{h}]=[f_2,r_2]^{\epsilon_2}\cdots [f_{k+1},r_{k+1}]^{\epsilon_{k+1}}$ which holds in~$F$. 
We recognize here the equation satisfied by the generators of a surface group of genus~$k+1$. Up to changing the order of the factors and to exchanging~$f_i$ and~$r_i$ we can rewrite it as $$[\tilde{g},\tilde{h}][f_2,r_2]\cdots [f_{k+1},r_{k+1}]=1.$$
Hence, let~$\Sigma$ be the closed orientable surface of genus~$k+1$, and define a morphism~$\tilde{\rho}:\pi_1(\Sigma)\to F$ by setting~$\tilde{\rho}(a_1)=\tilde{g},\tilde{\rho}(b_1)=\tilde{h}$ and~$\tilde{\rho}(a_i)=f_i,\tilde{\rho}(b_i)=r_i$ for~$1<i<k+1$. By Lemma~\ref{formenormale}, we may write~$\tilde{\rho}=f\circ i_*\circ \phi$, where~$\phi$ is an automorphism of~$\pi_1(\Sigma)$,~$i_*$ is induced by the inclusion~$\Sigma\to H_{k+1}$ and~$f:\pi_1(H_{k+1})=F_{k+1}\to F$ is a morphism.  Let~$\overline{\phi}:\Sigma\to \Sigma$ be a diffeomorphism of~$\Sigma$ fixing the base point and inducing the automorphism~$\phi$.
We set~$M=(H_{k+1}\setminus T)\cup_{\overline{\phi}}\overline{H}_{k+1}$. By our construction, the representation~$\tilde{\rho}:\pi_1(\Sigma)\to F$ induces a representation~$\rho:\pi_1(H_{k+1}\setminus T)\to G$. On the other hand, such a representation extends to~$\pi_1(M)$ if and only if~$\rho(\phi^{-1}(b_i))=1$ for all~$1\le i\le k+1$. In our case, this follows directly from the fact that~$i_*(b_i)=1$. 
We have shown that~$M$ is an overcommuting manifold for the pair~$(g,h)$ and has a Heegaard decomposition of genus $k+1$. This finally proves the theorem. 
\end{proof}

Let us prove now Lemma \ref{formenormale}. Although it is quite well-known (it is very similar to Lemma 3.2 in \cite{jaco} for instance), we include a proof for completeness.
\begin{lemma}\label{formenormale}
Let~$\Sigma$ be a surface bounding a standard handlebody~$H_{k+1}$ and let~$F$ be a free group. Then, any morphism~$\rho:\pi_1(\Sigma)\to F$ can be written as $\rho=f\circ i_*\circ \phi$, where~$\phi$ is an automorphism of~$\pi_1(\Sigma)$ preserving the orientation, $i_*:\pi_1(\Sigma)\to \pi_1(H_{k+1})$ is induced by the inclusion and~$f:\pi_1(H_{k+1})\to F$ is a morphism. 
\end{lemma}
\begin{proof}
One can suppose that~$F$ is isomorphic to~$\pi_1(X)$ for some graph~$X$. Moreover, one can find simplicial structures on~$\Sigma$ and~$X$ and a simplicial map $h:\Sigma\to X$ such that~$\rho=h_*$. Let~$E$ be the set of middles of the edges of~$X$. By transversality, ~$h^{-1}(E)$ is a collection of disjoint curves in~$\Sigma$. Hence, there exists a pants decomposition of~$\Sigma$ such that any connected component of~$h^{-1}(E)$ is parallel to a curve of the decomposition or homotopically trivial. Let~$H$ be a handlebody bounding~$\Sigma$ such that any component of the pants decomposition bounds a disc in~$H$. Our construction ensures that~$\rho$ factors through~$\pi_1(H)$ which is free. To conclude,  it remains to notice that any two handlebodies bounding~$\Sigma$ are related by an element of the mapping class group. 
\end{proof}

\begin{example}
In view of the overcommutation length, the simplest examples of overcommuting pairs~$(g,h)$ are those with~$\ocl(g,h) =0$. On one hand, they correspond to the unique 3-manifold~$M$ with toric boundary and Heegaard genus~$1$, which is the solid torus. On the other hand, considering a presentation $G=F/R$, they correspond to the case where the lifts $\tilde{g},\tilde{h}$ already commute in~$F$. This is only possible if they are powers of a same element in $F$, and hence in $G$. To sum up, we have the equivalence $$\ocl(g,h)=0\iff g=t^n\text{ and } h=t^m\text{ for some }t\in G\text{ and }n,m\in \Z.$$
\end{example}

Let us conclude this section with simple properties of the overcommutation length. 

\begin{proposition} Let $G$ be a group and $g,h$ be two overcommuting elements in $G$; 
\begin{enumerate}
\item For any morphism $\phi:G\to H$ we have $\ocl(\phi(g),\phi(h))\le \ocl(g,h)$. 
\item If $g$ overcommutes with $h_1,h_2\in G$ then $\ocl(g,h_1h_2)\le \ocl(g,h_1)+\ocl(g,h_2)+1$. 
\item The following stable overcommutation length is well-defined:
$$\socl(g,h)=\lim_{\min(m,n)\to\infty}\frac{\ocl(g^m,h^n)}{mn}.$$
\end{enumerate}
\end{proposition}
\begin{proof}
(1) Take two presentations $G=F/R$ and $H=F'/R'$. As $F$ is free, there exists a map $\Phi:F\to F'$ inducing $\phi$. We conclude as in the item (4) of Remark~\ref{biendef}.

(2) Take a presentation $G=F/R$ and lifts $\tilde g_1,\tilde h_1$ (respectively $\tilde g_2,\tilde h_2$) minimizing the overcommutation length of $g,h_1$ (respectively $g,h_2$). 
As $[\tilde g_1,\tilde h_1\tilde h_2]=[\tilde g_1,\tilde h_1]\tilde h_1 [\tilde g_1,\tilde h_2]\tilde h_1^{-1}$, we get $\ocl(g,h_1h_2)\le \ocl(g,h_1)+\ocl(\tilde g_1,\tilde h_2)$. As $\tilde g_1$ and $\tilde g_2$ differ by an element of $R$, we conclude as in the item (2) of Remark \ref{biendef} that $\ocl(\tilde g_1,\tilde h_2)\le \ocl(g,h_2)+1$ and the result follows. 

(3) The existence of the limit follows from the subadditivity (with defect 1) in both variables by a multivariate Fekete's lemma. 
\end{proof}
The stable overcommutation length enjoys the same properties as the overcommutation length like monotonicity and subadditivity in both variables (when defined). 

\begin{example}
Let $M$ be a 3-manifold with torus boundary and consider a morphism $\phi:\pi_1(M)\to \Z$ mapping $m$ to $1$ and $l$ to $0$. We set $M_n$ to be the cyclic cover of $M$ corresponding to the subgroup $\phi^{-1}(n\Z)$. This is again a 3-manifold with torus boundary and generators of $\pi_1(\partial M_n)$ are $m^n$ and $l$. 
We conclude that in $\pi_1(M)$, $\ocl(m^n,l)$ is less than the Heegaard genus $g(M_n)$ of~$M_n$ and hence that $\socl(m,l)\le \liminf\limits_{n\to\infty} \frac{g(M_n)}{n}$. 

Taking for $M$ a fibered manifold over the circle, we observe that the genus of~$M_n$ is bounded from above, hence $\socl(m,l)=0$ in that case. In the same way, the pairs of Example \ref{surfacecroixcercle} have trivial $\socl$ as there is a self-covering of~$\Sigma\times S^1$ which has index $n$ on the boundary. 
\end{example}

\section{Overcommutation  in the affine group}\label{OCMT}

 \subsection{Translations overcommute}\label{affine}
 Let $k$ be a field containing $\frac{1}{6}$ and denote by $\aff(k)$ the group of affine transformations of $k$. This group fits into an exact sequence 
 $$1\to k\to \aff(k)\to k^*\to 1$$ where $t\in k$ maps to the translation $x(t):z\mapsto z+t$. This sequence is split as there is a section mapping $u\in k^*$ to the homothety $h(u):z\mapsto uz$. 
 
We are interested here in the overcommutation of $x(s)$ and $x(t)$ for $s,t\in k$ and in their overcommutation length. Let us first prove that these elements overcommute. For that, we consider the presentation 
\begin{eqnarray*}
\aff(k)=\langle x(t),h(u), t\in k, u\in k^*&|& x(t+s)=x(t)x(s), h(uv)=h(u)h(v), \\
&&h(u)x(t)h(u)^{-1}=x(ut)\rangle.
\end{eqnarray*}
 
 Given $s,t\in k$, we can consider $F(s,t)=[x(s),x(t)]\in H_2(\aff(k),\Z)$. This expression is bilinear in $s$ and $t$ and, by conjugation with $h(u)$, it satisfies $F(s,t)=F(us,ut)$. If $u$ is an integer, we get $(u^2-1)F(s,t)=0$. Applying this to $u=2,3$ we conclude that $F(s,t)=0$. 
 
The same argument would work if we replace $s$ and $t$ by formal variables, that is if we consider $x(s)$ and $x(t)$ as translations in $\aff(k[s,t])$ which is the group of transformations of the form $z\mapsto \lambda z+P$ for $\lambda\in k^*$ and $P\in k[s,t]$. This suggests the following definition:

\begin{definition}
Given a field $k$ containing $\frac{1}{6}$, we define the overcommutation length of translations as the constant $\ocl(x(s),x(t))$ where $x(s),x(t)\in \aff(k[s,t])$. 
An overcommuting manifold for this pair will be called an overcommuting manifold for translations (OCMT for short).  
 \end{definition}

The proof given above can be translated into topological terms as in the following proposition.
\begin{proposition}
There exists an OCMT which is obtained by gluing copies of three different manifolds of the form $P_i\times S^1$ along their boundaries, where $P_i$ are (genus 0) surfaces. 
\end{proposition}
We observe that this proposition implies Proposition \ref{vol_OCMT} because the simplicial volume is additive under gluing along tori and the simplicial volume of  $P_i\times S^1$ vanishes. 
\begin{proof}
The key point is to translate equalities into cobordisms: for instance the equality $F(s,t)=F(us,ut)$ can be viewed as a cobordism $S^1\times S^1\times [0,1]\to B \aff(k[s,t])$ mapping $\{1\}\times\{1\}\times\{0,1\}$ to the base point and the paths $S^1\times \{1\} \times \{0\},\{1\}\times S^1\times\{0\},\{1\}\times\{1\}\times [0,1]$ to paths representing $x(s),x(t),h(u)$ respectively. The boundary of this cobordism is $F(us,ut)-F(s,t)$, viewed in $\Omega^2(BG)=H_2(G,\Z)$. 
In the same way, the equality $F(s,nt)=nF(s,t)$ can be obtained as the boundary of a map $S^1\times P_n\to B\aff(k[s,t])$ where $P_n$ is a disc with $n$ holes. 
The circle component is sent to $x(s)$, the boundary of the innermost circles of $P_n$ are sent to $x(t)$ whereas the outermost circle is sent to $x(nt)$. 
Translate the equality $F(2s,2t)=4F(s,t)$: by gluing three $P_2\times S^1$, we bound 5 tori $F(s,t)$, one negative, four positive. Glue one positive side on one negative side to get a manifold with three tori $F(s,t)$ as boundary. 
Translate now the equality $F(3s,3t)=9F(s,t)$: by gluing four $P_3\times S^1$, we bound 10 tori $F(s,t)$, one positive, 9 negative. Gluing the positive one on a negative one, we get 8 negative $F(s,t)$. 
By taking three copies of the first construction, we get 9 positive $F(s,t)$ that we glue with the 8 negative ones of the second construction. This leaves one $F(s,t)$ left and we are done. 
\end{proof} 
Notice that we used exactly 13 $P_i\times S^1$ for this construction, which will probably not give the optimal overcommutation length. Hence in the next section, we look for another strategy to obtain it.

\subsection{OCMT and Poincaré duality with coefficients}

The purpose of this section is to translate the properties of an overcommuting manifold for translations in terms of twisted (co)-homology. For a morphism $\lambda:\pi_1(M)\to k^*$, we will denote by $k_\lambda$ the vector space $k$ with the action of $\pi_1(M)$ given by $\gamma. x=\lambda(\gamma)x$. 

\begin{lemma}\label{pd}
An irreducible manifold $M$ with torus boundary is an overcommuting manifold for translations if and only if there exists a morphism $\lambda:\pi_1(M)\to k^*$ mapping $\pi_1(\partial M)$ to 1 such that one of the following equivalent conditions is verified:
\begin{enumerate}
\item[(i)]The map $H^1( M,k_\lambda)\to H^1(\partial M,k_\lambda)$ is surjective.\\
\item[(ii)]The map $H_1( \partial M,k_\lambda)\to H_1(M,k_\lambda)$ vanishes.\\
\item[(iii)]The map $H_1( \partial M,k_{\lambda^{-1}})\to H_1(M,k_{\lambda^{-1}})$ is injective.
\end{enumerate}
\end{lemma}
\begin{proof}
The exact sequence of pairs, Poincaré duality and the universal coefficient theorem with twisted coefficients yield the following commutative diagram where lines are exact:
$$\xymatrix{H_2(M,\partial M,k_\lambda)\ar[r]\ar[d]^\sim& H_1(\partial M,k_\lambda)\ar[r]\ar[d]^\sim&H_1(M,k_\lambda)\ar[d]^\sim\\
H^1(M,k_\lambda)\ar[r]\ar[d]^{\sim}& H^1(\partial M,k_\lambda)\ar[r]\ar[d]^\sim& H^2(M,\partial M,k_\lambda)\ar[d]^\sim\\
H_1(M,k_{\lambda^{-1}})^*\ar[r]& H_1(\partial M,k_{\lambda^{-1}})^*\ar[r]&H_2(M,\partial M,k_{\lambda^{-1}})^*.}$$
One reads on it the fact that the Properties (i),(ii),(iii) are equivalent.

Take a manifold with torus boundary and a representation $\rho:\pi_1(M)\to \aff(k[s,t])$ mapping $m$ and $l$ to $x(s)$ and $x(t)$ respectively. Writing $\rho(\gamma):z\mapsto\lambda(\gamma)z+\sum_{i,j\ge 0}a_{i,j}(\gamma)s^it^j$ we observe that $\lambda:\pi_1(M)\to k^*$ is a morphism mapping $m$ and $l$ to 1 whereas each coefficient $a_{i,j}:\pi_1(M)\to k$ is a $\lambda$-cocycle, meaning that it satisfies 
$$a_{ij}(\gamma\delta)=a_{ij}(\gamma)+\lambda(\gamma)a_{ij}(\delta), \quad \forall \gamma,\delta\in \pi_1(M).$$
This shows that $a_{ij}$ may be viewed as an element of $H^1(\pi_1(M),k_{\lambda})$ which is isomorphic to $H^1(M,k_{\lambda})$ by the irreducibility assumption. 
We also have $a_{10}(m)=1,a_{10}(l)=0, a_{01}(m)=0,a_{01}(l)=1$. This shows that the class of $a_{01}$ and $a_{10}$ restrict to a basis of $H^1(\partial M,k_\lambda)\simeq H^1(\partial M,k)$, proving Property (i). 

Reciprocally, given $\lambda$ and two cocycles $a_{10},a_{01}\in Z^1(\pi_1(M),k_\lambda)$ restricting to the standard basis on the boundary, one can set $\rho(\gamma):z\mapsto \lambda z+ a_{10}(\gamma)s+a_{10}(\gamma)t$ and observe that it satisfies all the required properties. 
\end{proof}

\subsection{OCMT and Alexander modules}
The goal of this section is to exhibit an overcommuting manifold for translations of Heegaard genus 3. To this end, we restrict to a particular class of representations, where formulating the property of being an overcommuting manifold for translations becomes tangible. More precisely, we suppose in the sequel that $\lambda$ is the composition of a morphism $\phi:\pi_1(M)\to\Z$ vanishing on $\pi_1(\partial M)$ with the morphism $\ev_u:\Z\to k^*$ mapping $1$ to $u\in k^*$. This allows to translate the problem in terms of the $\Lambda$-module $H_1(M,\Lambda)$ where $\Lambda=k[t^{\pm 1}]$ and $\gamma\in \pi_1(M)$ acts by multiplication with $t^{\phi(\gamma)}$. 
\begin{lemma}\label{lambda}
The manifold $M$ is an OCMT for $\lambda=\ev_u\circ \phi$ if and only if the image of $H_1(\partial M,\Lambda)$ lies in $(t-u)H_1(M,\Lambda)$. Moreover we have then:
\begin{enumerate}
\item $H_1(M,\Lambda)$ is not cyclic.
\item For any Dehn filling $N$ of $M$, $H_1(N,\Lambda)/(t-u^{-1})H_1(N,\Lambda)\ne 0$. 
\end{enumerate}
\end{lemma}
\begin{proof}

Considering $k$ as the module $\Lambda/(t-u)\Lambda$, one gets the following commutative diagram where lines correspond to the universal coefficient theorem:

$$\xymatrix{0\ar[r]&H_1(\partial M,\Lambda)\otimes k\ar[r]\ar[d]^\alpha& H_1(\partial M,k_{\lambda})\ar[r]\ar[d]^\beta& \operatorname{Tor}(H_0(\partial M,\Lambda),k)\ar[d]\ar[r]&0\\
0\ar[r]&H_1(M,\Lambda)\otimes k\ar[r]& H_1(M,k_{\lambda})\ar[r]& \operatorname{Tor}(H_0(M,\Lambda),k)\ar[r]&0.}$$

As $\phi$ is trivial on $\pi_1(\partial M)$, we have $H_*(\partial M,\Lambda)=H_*(\partial M,\Z)\otimes\Lambda$ and the upper right group vanishes. This shows that $\alpha$ vanishes if and only if  $\beta$ vanishes, which is equivalent to being an OCMT by Lemma \ref{pd}.  The equivalence follows. 
Let us consider now $k=\Lambda/(t-u^{-1})\Lambda$ and the same commutative diagram. By Lemma \ref{pd}, the map $\beta$ is injective implying that the map $\alpha$ is also injective.
Hence $H_1(M,\Lambda)\otimes k$ is at least 2-dimensional and the first property follows. 
For the second point, let $\gamma$ be the curve on $\partial M$ bounding a disc in $N$. By a Mayer-Vietoris argument we have that $H_1(N,\Lambda)=H_1(M)/\Lambda \gamma$ where $\gamma$ stands for the image of $\gamma$ in $H_1(M,\Lambda)$. Look at the previous commutative diagram again for $k=\Lambda/(t-u^{-1})\Lambda$. We observe that by the injectivity of $\beta$, the image of $\gamma$ in $H_1(M,\Lambda)\otimes k$ does not generate this group. Hence $H_1(N,k_{\lambda^{-1}})$ cannot be trivial and the conclusion follows. 
\end{proof}

This lemma says that if $M$ is the complement of a knot $K$ in a manifold $N$ where there exists $\phi:\pi_1(N)\to \Z$ such that $H_1(N,\Lambda)$ is a torsion module, then the Alexander polynomial $\Delta$ of $N$ should satisfy $\Delta(u^{-1})=0$ otherwise it would contradict Property (2) of Lemma \ref{lambda}. In particular $\Delta$ cannot be trivial. If $u$ and $u^{-1}$ were conjugated by an automorphism of $k$, the maps $H_1(\partial M,k_\lambda)\to H_1(M,k_\lambda)$ where $\lambda=\ev_u\circ\phi$ and $\lambda=\ev_{u^{-1}}\circ \phi$ should be isomorphic. As one is zero and the other is injective, this is impossible: this means that $\Delta$ has to be reducible. More precisely, one can write $\Delta=PQ$ with $P(u^{-1})=0$ and $P(u)\ne 0$. The simplest knot in $S^3$ satisfying this property is the stevedore knot $6_1$ with $u=2$. 

In Figure \ref{stevedore}, we give such a manifold by showing three surgery pictures of it. The first one shows that the equivariant linking matrix of the link $L\cup K$ in the complement of the unknot $U$ is $\begin{pmatrix} 2t-5+2t^{-1} & t-2 \\ t^{-1}-2 & 0\end{pmatrix}$. This proves that $H_1(M,\Lambda)=\Lambda m_L\oplus \Lambda m_K/ \Lambda l_L$ with $l_L=(2t-5+2t^{-1})m_L+(t^{-1}-2)m_K$ and $l_K=(t-2)m_L$. These expressions show that $m_K$ and $l_K$ are divisible by $(t-2)$, hence $M$ is an OCMT. 

The second figure is the nicest picture of $M$ whereas the third one shows that $M$ is the complement of a knot in the 0-surgery over Stevedore's knot. One can also deduce from the last picture that the Heegaard genus of $M$ is at most 3. 
\begin{figure}[htbp]
\begin{center}
 \def\svgwidth{12cm}
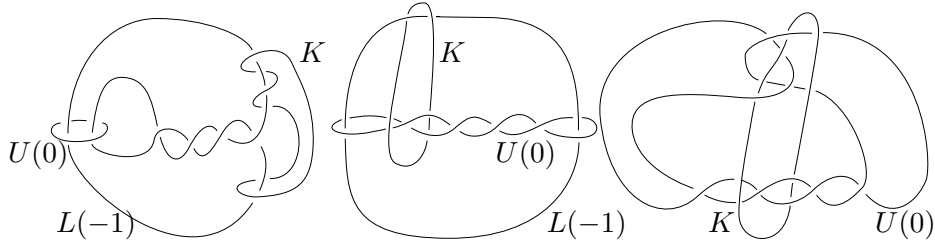
\caption{A surgery presentation of an OCMT}\label{graph}\label{stevedore}
\end{center}
\end{figure}
In the following lemma we show that there is no overcommuting manifold for translations of Heegaard genus 2. Together with our construction of such a manifold of Heegaard genus 3, this proves the minimality statement in Proposition \ref{oclaffine}.
\begin{lemma}
There is no overcommuting manifold for translations of Heegaard genus 2. 
\end{lemma}
\begin{proof}
Suppose that $M$ is such a manifold. It is obtained by gluing a $2$-handle to a standard handlebody $H_2$ along a curve $\gamma\subset \Sigma_2=\partial H_2$. As $M$ has torus boundary, $\gamma$ is non-separating. As $\lambda$ is trivial on the boundary of $M$, it is trivial on $\Sigma_2\setminus \gamma$. We conclude that for $\delta$ in $\pi_1(\Sigma_2)$ we have $\lambda(\delta)=u^{\gamma\cdot\delta}$ where $\cdot$ denotes the algebraic intersection number and $u$ is in $k^*$. 

Consider then the map $\phi:H_1(\Sigma_2,\Z)\to \Z$ given by $\phi(\delta)=\gamma\cdot\delta$. By the preceding discussion, it should extend to $H_1(M,\Z)$ as a non trivial map - this means that $\gamma$ is in the kernel of the inclusion map $H_1(\Sigma_2,\Z)\to H_1(H_2,\Z)$. This allows to consider coefficients in $\Lambda$ and use the analysis of Lemma \ref{lambda}. 
As $\phi$ is non trivial, we have $H^0(M,\Lambda)=H^0(H_2,\Lambda)=0$. By retracting $H_2$ on a wedge of two circles, we observe that the homology of $H_2$ with coefficients in $\Lambda$ can be computed from a complex of the form $ \Lambda\overset{\partial}{\leftarrow} \Lambda^2$. The map $\partial $ is necessarily surjective as $H_0(H_2,\Lambda)=0$ and its kernel is isomorphic to $\Lambda$. That is, $H_1(H_2,\Lambda)$ is isomorphic to $\Lambda$. 

From the exact sequence of the pair $(M,H_2)$, we get that $H_1(M,\Lambda)$ is a quotient of $H_1(H_2,\Lambda)$, hence it is cyclic. This contradicts Lemma \ref{lambda} as $H_1(M,\Lambda)$ cannot be cyclic if $M$ were an OCMT. 
\end{proof}

\begin{remark}\label{uppersubgroup}
Assume~$\sqrt{2}\in k$. We can apply the results of this section to the subgroup of upper triangular matrices in $\SL_2(k)$ that are of the form~$\left(\begin{smallmatrix} \sqrt{2}^n & \sqrt{2}^{-n}t \\ 0 & \sqrt{2}^{-n}\end{smallmatrix}\right)$, for~$t\in k$ and~$n\in\Z$. 
This subgroup is isomorphic to the group of affine transformations of the the form $z\mapsto 2^nz+t$, for~$t\in k$ and~$n\in\Z$. The above example of an OCMT is constructed from the map $\lambda=\ev_2\circ \phi$, so after specializing the variables~$s$ and~$t$, the representation indeed takes values in the affine transformations of the form $z\mapsto 2^nz+t$, for~$t\in k$ and~$n\in\Z$. 
\end{remark}

\section{Effective overcommutation}\label{effectif}
\subsection{Contracting presentations}
\begin{definition}
Let $1\to R\to F\to G\to 1$ be a presentation of $G$ and fix a set $S=\{r_i, i\in I\}$ generating $R$ normally. 
\begin{enumerate}
\item If $x\in [F,F]$, we define $\cl(x)=\inf\{k\in \N, x=\prod_{i=1}^k [f_i,g_i]\}$, where $f_i,g_i\in F$. 
\item If $x\in R$, we define $l_S(x)=\inf\{k\in \N, x=\prod_{j=1}^k f_j r_{i_j}^{\pm 1}f_j^{-1}\}$, where $i_j\in I$, and $f_j\in F$.
\item If $x\in [F,R]$, we define $\cl_R(x)=\inf\{k\in \N, x=\prod_{i=1}^k [f_i, r_{i}]^{\pm 1}\}$, where $f_i\in F$ and $r_i\in R$.
\item We will say that the presentation is $(C,C')$-contracting if $H_2(G,\Z)=0$ and for any $x\in [F,R]=[F,F]\cap R$ we have 
$$\cl_R(x)\le C l_S(x)+C'\cl(x).$$
\end{enumerate}
\end{definition}
Notice that the quantity $l_S(x)$ is sometimes called the area of the relation $x$. 
\begin{example}\label{one-relator}
Let $G=\langle a_1,\ldots,a_n|r\rangle$ be a presentation such that the abelianization of $r$ is non-zero in $\Z^n$. Then it is $(1/2,0)$-contracting. 

Indeed, take $x\in [F,F]\cap R$ and write $x=\prod_{j=1}^{l_S(x)}f_j r^{\epsilon_j}f_j^{-1}$. As $x$ is a product of commutators, its abelianization vanishes. As $r\ne 0$, this implies $\sum_j \epsilon_j=0$ and we can write $l_S(x)=2k$. Consider two consecutive terms with opposite signs such as $frf^{-1}gr^{-1}g^{-1}$. We may write it $f[r,f^{-1}g]f^{-1}$. Repeating the argument, we still find subwords of the form $frf^{-1}xgr^{-1}g^{-1}$ where $x\in [F,R]$. We replace it with $f[r,f^{-1}xg]f^{-1}x$ which creates one more commutator.  This proves that $\cl_R(x)\le k$ as claimed.
\end{example}
We may consider the example of the torus knot of parameters $(p,q)$ where $p$ and $q$ are two coprime integers. Its presentation is $G=\langle a,b|r\rangle$ with $r=a^pb^{-q}$. The meridian is $m=a^ub^v$ where $qu+pv=1$ and a longitude is $a^p$. We compute $[m,l]=a^ub^v a^p b^{-v}a^{-u}a^{-p}=a^ub^vrb^{q-v}a^{-u-p}=a^ub^vrb^{-v}r^{-1}a^{-u}$. This shows $l_S([m,l])=2$ and $\cl_R([m,l])=1$. This is coherent with the fact that the tunnel number of the torus knot complement is 1, and hence its Heegaard genus is 2. 
 
\subsection{Steinberg group}
Let $k$ be a field containing $\frac{1}{6}$. We define the Steinberg group $\St_2(k)=F/R$ where $F$ is the free group generated by the symbols $x_\alpha(t)$ where $\alpha=\pm 1$ and $t\in k$. The subgroup $R$ is normally generated by $r^1_\alpha(s,t)$ and $r^2_\alpha(u,t)$ where 
\begin{eqnarray*}
r_\alpha^1(s,t)&=&x_\alpha(s+t)x_\alpha(s)^{-1}x_\alpha(t)^{-1},\quad s,t\in k\\
r_\alpha^2(u,t)&=&w_\alpha(u)x_\alpha(t)w_\alpha(u)^{-1}x_{-\alpha}(u^{-2}t),\quad  u\in k^*, t\in k
\end{eqnarray*}
and where we have set $w_\alpha(u)=x_\alpha(u)x_{-\alpha}(-u^{-1})x_\alpha(u)$. We will denote below by $S$ the set of all Steinberg relations. 
The main property of this group is that it is the universal central extension of $\SL_2(k)$. Explicitly, the map $\pi:\St_2(k)\to \SL_2(k)$ defined by 
$$\pi(x_1(t))=\begin{pmatrix} 1 & t \\ 0 & 1\end{pmatrix}\text{ and }\pi(x_{-1}(t))=\begin{pmatrix} 1 & 0 \\ t & 1\end{pmatrix}$$
is surjective, its kernel is central and $H_1(\St_2(k),\Z)=H_2(\St_2(k),\Z)=0$. 

Let us recall the statement of Theorem \ref{T2} that we aim to prove now.
\begin{theorem*}
If $\frac{\sqrt{2}}{6}\in k$, the presentation of the Steinberg group is $(5,2)$-contracting. 
\end{theorem*}
Before starting the proof, we collect some well-known facts about the Steinberg group, see \cite{steinberg,milnor}. 

\begin{proposition}\label{calculs}
Set $h_\alpha(u)=w_\alpha(u)w_\alpha(1)^{-1}$ and for any $u,v\in k^*$, $c(u,v)=h_\alpha(uv)h_\alpha(u)^{-1}h_\alpha(v)^{-1}$. The following identities hold in $\St_2(k)$:
\begin{enumerate}
\item $w_\alpha(u)=w_\alpha(-u)^{-1}=w_{-\alpha}(-u^{-1})$
\item $h_\alpha(u)x_\alpha(t)h_\alpha(u)^{-1}=x_\alpha(u^2t)$
\item $h_\alpha(t)=h_{-\alpha}(t)^{-1}$
\item $w_\alpha(u)w_{-\alpha}(v)w_\alpha(u)^{-1}=w_{\alpha}(-u^2t)$
\item $c(u,v)$ is central in $\St_2(k)$. 
\end{enumerate}
\end{proposition}

\subsection{Proof of Theorem \ref{T2}}
Fix $a$ an integer which is distinct from $0,1,-1$ in $k$. By equation (2) of Proposition \ref{calculs}, we have for all $\alpha\in \{\pm 1\}$ and $t\in k$ the equality
\begin{equation}\label{commutator}
x_\alpha(t)=[h_\alpha(a),x_\alpha(t/(a^2-1))].
\end{equation}

This shows that $H_1(\St_2(k),\Z)$ vanishes and is a key ingredient in the proof by the following argument. 
Let $\psi:F\to [F,F]$ be the morphism mapping $x_\alpha(t)$ to $[h_\alpha(a),x_\alpha(t/(a^2-1))]$ and suppose that we found a constant $C>0$ such that for all $r\in S$, $\psi(r)\in [F,R]$ and $\cl_R(\psi(r))\le C$. 

Then we pick $x=\prod_{i=1}^k[f_i,g_i]=\prod_{j=1}^lh_j r_{j}^{\pm 1}h_j^{-1}\in [F,F]\cap R$ where $f_i,g_i,h_j\in F$ and $r_j\in S$ so that $k=\cl(x)$ and $l=l_S(x)$. Applying $\psi$ we get on one hand $\psi(x)=\prod_{j=1}^l\psi(h_j) \psi(r_{j})^{\pm 1}\psi(h_j)^{-1}$ and hence $\cl_R(\psi(x))\le Cl_S(x).$

On the other hand, we can bound $\cl_R(x)$ in terms of $\cl_R(\psi(x))$ in the following way.
Write $\psi(x)=\prod_{i=1}^k [\psi(f_i),\psi(g_i)]$ and observe that for all $g\in F$ one has $\psi(g)=gr(g)$ for some $r(g)\in R$ by equation \eqref{commutator}. The formula of Remark \ref{biendef} shows that $[\psi(g),\psi(h)]=\xi [g,h]\xi'$ where $\xi$ and $\xi'$ denote single commutators in $[F,R]$.
As for any $x,f\in F$ and $r\in R$ one has $x[f,r]=[xfx^{-1},xrx^{-1}]x$, one can also write $[\psi(g),\psi(h)]=\xi \xi'' [g,h]$. 
Applying this to each factor produces $2k$ commutators which can be moved by applying the above trick. This implies that $\psi(x)x^{-1}$ can be written using $2k$ commutators in $[F,R]$ and the conclusion follows. Observe that we obtained along the way that $H_2(\St_2(k),\Z)=0$. 

We get finally $\cl_R(x)\le \cl_R(\psi(x))+2k\le Cl_S(x)+2\cl(x)$. Hence, the presentation is $(C,2)$-contracting. In order to prove the theorem it remains to show the existence of $C$. We show below that we can take $C=5$. 

\subsubsection{First Steinberg relation} 

To save space, we write $h=h_\alpha(a), t'=t/(a^2-1),s'=s/(a^2-1)$. We have by definition
$$\psi( r_\alpha^1(s,t))=[h,x_\alpha(s'+t')][h,x_\alpha(s')]^{-1}[h,x_\alpha(t')]^{-1}.$$
In the sequel, $\xi$ denotes an arbitrary single commutator in $[F,R]$: for instance $\xi^2$ is an arbitrary product of two commutators. 
As $x_\alpha(s'+t')=x_\alpha(s')x_\alpha(t')$ modulo $R$, we can write $\psi(x_\alpha(s+t))=[h,x_\alpha(s')x_\alpha(t')]\xi=\xi[h,x_\alpha(s')x_\alpha(t')]$. Hence, 
\begin{eqnarray*}
\psi(r_\alpha^1(s,t))&=&\xi[h,x_\alpha(s')]x_\alpha(s')[h,x_\alpha(t')]x_\alpha(s')^{-1}[h,x_\alpha(s')]^{-1}[h,x_\alpha(t')]^{-1}\\ 
&=&\xi\Big [[h,x_\alpha(s')]x_\alpha(s'),[h,x_\alpha(t')]\Big]=\xi^3\Big[x_\alpha(\frac{a^2}{a^2-1}s),x_\alpha(t)\Big],
\end{eqnarray*}
where in the last equality, we used Equation \eqref{commutator} in both arguments of the commutator. 
Using Lemma~\ref{steinbergtranslations} below, we finally get
$$\cl_R(\psi(r_\alpha^1(s,t)))\le 5.$$

\begin{lemma}\label{steinbergtranslations}
$[x_\alpha(s),x_\alpha(t)]=\xi^2$ for all~$s,t\in k$. 
\end{lemma}

\begin{proof}
It suffices to give a proof for~$\alpha=1$.
Let~$U$ be the subset of~$\St_2(k)$ that consists of all the elements of the form~$x_1(t)h_1(\sqrt 2)^n,$ where~$t\in k$ and~$n\in \Z$. 
Using equation (2) of Proposition~\ref{calculs}, one can directly see that~$U$ is a subgroup of~$\St_2(k)$. Furthermore, the projection 
homomorphism~$\pi: \St_2(k)\to \SL_2(k)$ sends the element~$x_1(t)h_1(\sqrt 2)^n$ 
to the matrix~$\left(\begin{smallmatrix} \sqrt{2}^n & \sqrt{2}^{-n}t \\ 0 & \sqrt{2}^{-n}\end{smallmatrix}\right)$. It is apparent that~$\pi|_U$ is an isomorphism 
onto its image, so~$U$ is isomorphic to the subgroup of~$\SL_2(k)$ consisting of all matrices 
of the form~$\left(\begin{smallmatrix} \sqrt{2}^n & \sqrt{2}^{-n}t \\ 0 & \sqrt{2}^{-n}\end{smallmatrix}\right)$, for~$t\in k$ and~$n\in\Z$. 
We now obtain the result by using our example of an OCMT of Heegaard genus~$3$ from Section~\ref{OCMT}, compare with Remark~\ref{uppersubgroup}. 
\end{proof}

\subsubsection{Second Steinberg relation}\label{ssr}

We compute 
$$\psi(r_\alpha^2(u,t))=\psi(w_\alpha(u))[h_\alpha(a),x_\alpha(t')]\psi(w_\alpha(u))^{-1}[h_{-\alpha}(a),x_{-\alpha}(u^{-2}t')].$$
We can insert the conjugation by $\psi(w_\alpha(u))$ inside the commutator and observe that we have  $\psi(w_\alpha(u))=w_\alpha(u)$ modulo $R$. Moreover, 
\begin{eqnarray*}
w_\alpha(u)h_\alpha(a)w_\alpha(u)^{-1}&\overset{(1)}{=}&w_{-\alpha}(-u^{-1})w_{\alpha}(a)w_\alpha(-1)w_{-\alpha}(-u^{-1})^{-1}\\
&\overset{(4)}{=}&w_{-\alpha}(-u^{-2}a)w_{-\alpha}(u^{-2})\!=\!h_{-\alpha}(-u^{-2}a)h_{-\alpha}(-u^{-2})^{-1}\\
&\overset{(5)}{=}&h_{-\alpha}(a)\mod Z(\St_2(k)).
\end{eqnarray*}
Here, we wrote above the equal signs the the equations of Proposition~\ref{calculs} that we used. 
On the other hand we have in $\St_2(k)$: $w_\alpha(u)x_\alpha(t')w_\alpha(u)^{-1}=x_{-\alpha}(-u^{-2}t')$. Writing $h=h_{-\alpha}(a)$, we have for some $r\in R$ and $c\in F$ mapping into $Z(\St_2(k))$: 
\begin{eqnarray*}
\psi(r_\alpha^2(u,t))&=&[hc,x_{-\alpha}(-u^{-2}t')r][h,x_{-\alpha}(u^{-2}t')]\\
&=&h[c,x_{-\alpha}(-u^{-2}t')r]h^{-1}[h,x_{-\alpha}(-u^{-2}t')r][h,x_{-\alpha}(u^{-2}t')].
\end{eqnarray*}
We first treat the term $[c,x_{-\alpha}(-u^{-2}t')r]=[c,[h,x_{-\alpha}(-u^{-2}t'')]]\xi$. 
From the Hall-Witt identity, we have for any $x,y\in F$: 
$$[[x,y],x^{-1}cx][[c,x],c^{-1}yc][[y,c],y^{-1}xy]=1.$$
As $c$ maps to the center, we have $[c,x],[y,c]\in R$ and $x^{-1}cx=c\mod R$. This proves $[[x,y],c]=\xi^3$. 

The remaining term gives $[h,x_{-\alpha}(-u^{-2}t')r][h,x_{-\alpha}(u^{-2}t')]=\xi[h,x^{-1}][h,x]$ where $x=x_{-\alpha}(u^{-2}t')$. 
As $1=[h,x^{-1}x]=[h,x^{-1}]x^{-1}[h,x]x,$ we get $[h,x^{-1}][h,x]=[x^{-1},[h,x]^{-1}]$. But modulo $R$ we have $[h,x]=x_{-\alpha}(u^{-2}t)=x^{a^2-1}$. 
This finally gives $[h,x^{-1}][h,x]=\xi$ and $\cl_R(\psi(r_\alpha^2(u,t)))\le 5$.

\subsection{Application to Ghys' example}\label{ghys}

Take $k=\Q(\sqrt{2},u)$ and set $v=1-u$. It is a well-known consequence of the Steinberg relation $c(u,v)=0$ that the elements $h_\alpha(u)$ and $h_\alpha(v)$ commute in $\St_2(k)$. We follow the proof of Lemma 9.8 in \cite{milnor} by keeping track of the Steinberg relations that were used. Replacing $h_\alpha(u)$ by the equivalent $\eta_\alpha(u)=w_\alpha(u)w_{-\alpha}(1)$, this gives 
$$[\eta_\alpha(u),\eta_\alpha(v)]=R_1(u,v)w_\alpha(-uv)^{-1}R_2(u,v)R_2(v,u)^{-1}w_\alpha(-uv)R_1(v,u)^{-1}$$
where $r^3_\alpha(u,t)=w_\alpha(u)x_\alpha(t)w_\alpha(-u)x_{-\alpha}(u^{-2}t)$,

\begin{eqnarray*}
R_1(u,v)&=&x_\alpha(-uv)^{-1}r^1_\alpha(u^2,-u)x_{-\alpha}(u^{-1})^{-1}x_\alpha(-u)^{-1}r^3_\alpha(-u,u^2)\\
&&r^1_\alpha(-v,-u)x_\alpha(-u)x_{-\alpha}(u^{-1})x_\alpha(-uv)\text{ and}\\ 
R_2(u,v)&=&x_\alpha(-uv)r^1_{-\alpha}(u^{-1},v^{-1}) x_\alpha(-uv)^{-1}r^1_\alpha(-v,v^2)w_\alpha(v)^{-1}\\
&&r^3_\alpha(v,v^2)w_\alpha(v).
\end{eqnarray*}
This computation shows that the commutation of $\eta_\alpha(u)$ and $\eta_\alpha(v)$ uses 12 relations and Theorem \ref{T2} gives that $\ocl(\eta_\alpha(u),\eta_\alpha(v))$ is at most 62 and hence the Heegaard genus of Ghys' example is at most 63. 

One can also proceed in the following more clever way. Denoting $\eta_\alpha'(u)=R_1(u,v)^{-1}\eta_\alpha(u)R_1(u,v)$ we have 
\begin{eqnarray*}
[\eta'_\alpha(u),\eta'_\alpha(v)]&=& w_\alpha(-uv)^{-1}R_2(u,v)R_2(v,u)^{-1}w_\alpha(-uv)R_1(v,u)^{-1}R_1(u,v)\\
&=&\xi R_2(u,v)R_2(v,u)^{-1}R_1(v,u)^{-1}R_1(u,v).
\end{eqnarray*}
On one hand $R_1(v,u)^{-1}R_1(u,v)$ is equal to 
$$\xi^3r^1_\alpha(-u,-v)^{-1}r^3_\alpha(-v,v^2)^{-1}r^1_\alpha(v^2,-v)^{-1}r^1_\alpha(u^2,-u)r^3_\alpha(-u,u^2)r^1_\alpha(-v,-u)$$
and $R_2(u,v)R_2(v,u)^{-1}$ is equal to 
$$\xi^2 r^4_{\alpha}(u,v)r^1_\alpha(-v,v^2)r^3_\alpha(v,v^2)r_\alpha^3(u,u^2)^{-1}r^1_\alpha(-u,u^2)^{-1}r^4_\alpha(v,u)^{-1},$$
where we have set $r^4_\alpha(u,v)=x_\alpha(-uv)r^1_{-\alpha}(u^{-1},v^{-1})x_\alpha(-uv)^{-1}$. 

We observe that relations come in pairs in this expression: we have for instance $r^1_\alpha(t,s)^{-1}r^1_\alpha(s,t)=[x_\alpha(s),x_\alpha(t)]=\xi^2$. In the same way, 
\begin{eqnarray*}
r^3_\alpha(s,t)r^3_\alpha(-s,t)^{-1}&=&w_\alpha(s)x_\alpha(t)w_\alpha(-s)w_\alpha(s)^{-1}x_\alpha(t)^{-1}w_\alpha(-s)^{-1}\\
&=& w_\alpha(s)[x_\alpha(t),w_\alpha(-s)w_\alpha(s)^{-1}][w_\alpha(-s),w_\alpha(s)^{-1}]w_\alpha(s)^{-1}.
\end{eqnarray*}
As $w_\alpha(-s)=w_\alpha(s)^{-1}$ modulo $R$, the second commutator can be replaced by $\xi$. We observe moreover that $w_\alpha(-s)w_\alpha(s)^{-1}$ maps to the center of $\St_2(k)$, hence by the same argument as in Subsection \ref{ssr}, we find $r^3_\alpha(s,t)r^3_\alpha(-s,t)^{-1}=\xi^5$. 
Finally, $[\eta'_\alpha(u),\eta'_\alpha(v)]=\xi^{24}R'_2R'_1$ where $R'_2R'_1$ is an expression using $6$ elements of $R$ once and their inverse. By moving the 4 middle terms in $R_2$ in the middle of $R_1$ we get directly $R'_2R'_1=\xi$ and $\ocl(h_\alpha(u),h_\alpha(v))\le 25$. 

\begin{remark}
We may generalise Ghys' example to any $A$-polynomial in the following way. Suppose that $P\in\overline{\Q}[x,y]$ is an irreducible polynomial and set $k=\operatorname{Frac}(\overline{\Q}[x,y]/(P))$. By Matsumoto's theorem, the elements $h_\alpha(x)$ and $h_\alpha(y)$ commute in $\St_2(k)$ (or their images in $\SL_2(k)$ overcommute) if and only if their commutator can be written in terms of the five Matsumoto's relations given for instance in \cite{Hutchinson}, Proposition 3.5. Among them, the most complicated is $c(u,(1-u)v)c(u,v)^{-1}$ which for $v=1$ is the one we dealt with in this section. 
Hence, in the general case, we can bound the complexity of a manifold $M$ such that $P$ divides its $A$-polynomial by the number of Matsumoto's relations, but we do not make it explicit here. 
\end{remark}

We have discussed how Ghys manifolds provide a topological reason for the Steinberg relation. 
Let us end this article by a discussion of other properties of Ghys manifolds, starting with a formal definition.

\begin{definition}
A Ghys manifold $M$ is an irreducible compact oriented 3-manifold with $\partial M=S^1\times S^1$ together with a representation $\rho:\pi_1(M)\to \SL_2(\Q(x))$ such that 
$$\rho(m)=\begin{pmatrix} x & 0 \\ 0 & x^{-1} \end{pmatrix}\text{ and }\rho(l)=\begin{pmatrix} (1-x) & 0 \\ 0 & (1-x)^{-1} \end{pmatrix}.$$
\end{definition}
Notice that it is easy to satisfy the irreducibility condition by removing prime factors of a non-irreducible example. 
\begin{proposition}
Suppose that $M$ is a Ghys manifold $M$. Then
\begin{enumerate}
\item If $M$ does not contain closed incompressible surfaces, no Dehn filling of $M$ gives $S^3$.
\item If $M$ is hyperbolic, the character variety of $M$ has at least three components.
\end{enumerate}
\end{proposition}
\begin{proof}
(1) By standard Culler-Shalen theory (see \cite{shalen}), the condition that $M$ does not contain closed irreducible surfaces implies that the restriction map $r:X(M)\to X(\partial M)$ is proper. In particular, its image is a Zariski-closed subset. 

Let $(p,q)$ be the slope of the surgery producing $S^3$, assuming that $p$ and $q$ are coprime. We have then $M=S^3\setminus V(K)$ where $V(K)$ denotes a tubular neighborhood of some knot $K$ in $S^3$. 
Let $(x,y)$ be a solution of the system $x+y=1$ and $x^py^q=1$. As $r$ is proper, there exists a representation $\rho:\pi_1(M)\to \SL_2(\overline{\Q})$ such that $\rho(m)$ and $\rho(l)$ have eigenvalues $(x,x^{-1})$ and $(y,y^{-1})$ respectively. If $x\ne \pm 1$ or $y\ne \pm 1$, we can suppose that $\rho(m)$ and $\rho(l)$ are diagonal. In particular, the curve $\gamma$ with slope $(p,q)$ in the boundary of $M$ satisfies $\rho(\gamma)=1$. 
This means that $\rho$ extends to $S^3$ and hence is trivial, contradicting the fact that $x\ne \pm 1$ or $y\ne \pm 1$. 

We observe that the equation $x+y=1$ forbids $x=\pm 1$ and $y=\pm 1$, hence one of the coordinates has to vanish, which forces $p=0$ or $q=0$. 
Let us suppose then by symmetry that $p=0$ and $q=1$: this means that the variable $x$ corresponds to the meridian of $K$ and $y$ to the longitude. 

Let $A(x,y)$ be the $A$-polynomial of $M$ and suppose that $A(-1,y)=0$ with $y\notin\{-1,0,1\}$. Again as $r$ is proper, there exists a representation $\rho:\pi_1(M)\to \SL_2(\overline{\Q})$ such that $\rho(l)$ is diagonal with entries $y,y^{-1}$. As $\rho(m)$ has eigenvalues $-1$ and commutes with $\rho(l)$, $\rho(m)=-\id$. As $m$  normally generates $\pi_1(M)$, this implies that $\rho$ is central, contradicting the assumption $x\ne \pm 1$. This gives $A(-1,y)\ne 0$. But as $x+y-1$ divides $A(x,y)$, we have $A(-1,2)=0$, which contradicts what we have just shown. This proves the first point of the proposition.

(2) If $M$ is hyperbolic, any representation $\rho:\pi_1(M)\to \SL_2(\overline{\Q})$ which lifts the holonomy representation is parabolic at the boundary, meaning that its parameters $x,y$ satisfy $x=\pm 1, y=\pm 1$. For any such $(x,y)$ we have $x+y\ne 1$: this implies that the component of $X(M)$ where this representation lives is distinct from the one projecting to the curve $x+y=1$. Recalling that there is also a curve of abelian representations, we just have proved that there are at least three disctinct components. 
\end{proof}

When writing this article, we noticed that the fundamental group of a Ghys manifold acts on a CAT(0) cube complex by reproducing the construction of Bass-Serre tree with three valuations (corresponding to $0,1,\infty$) instead of one. We hope to address this question in a forthcoming publication.

\end{document}

%% file: stevedore.pdf_tex
\begingroup%
  \makeatletter%
  \providecommand\color[2][]{%
    \errmessage{(Inkscape) Color is used for the text in Inkscape, but the package 'color.sty' is not loaded}%
    \renewcommand\color[2][]{}%
  }%
  \providecommand\transparent[1]{%
    \errmessage{(Inkscape) Transparency is used (non-zero) for the text in Inkscape, but the package 'transparent.sty' is not loaded}%
    \renewcommand\transparent[1]{}%
  }%
  \providecommand\rotatebox[2]{#2}%
  \ifx\svgwidth\undefined%
    \setlength{\unitlength}{934.89595188bp}%
    \ifx\svgscale\undefined%
      \relax%
    \else%
      \setlength{\unitlength}{\unitlength * \real{\svgscale}}%
    \fi%
  \else%
    \setlength{\unitlength}{\svgwidth}%
  \fi%
  \global\let\svgwidth\undefined%
  \global\let\svgscale\undefined%
  \makeatother%
  \begin{picture}(1,0.26461366)%
    \put(0,0){\includegraphics[width=\unitlength,page=1]{stevedore.pdf}}%
    \put(-0.00885678,0.08808635){\color[rgb]{0,0,0}\makebox(0,0)[lb]{\smash{$U(0)$}}}%
    \put(0.52567058,0.08808635){\color[rgb]{0,0,0}\makebox(0,0)[lb]{\smash{$U(0)$}}}%
    \put(0.31030601,0.19932866){\color[rgb]{0,0,0}\makebox(0,0)[lb]{\smash{$K$}}}%
    \put(0.46338209,0.19932866){\color[rgb]{0,0,0}\makebox(0,0)[lb]{\smash{$K$}}}%
    \put(0.75872109,0.01107243){\color[rgb]{0,0,0}\makebox(0,0)[lb]{\smash{$K$}}}%
    \put(0.58166946,0.01107243){\color[rgb]{0,0,0}\makebox(0,0)[lb]{\smash{$L(-1)$}}}%
    \put(0.04388772,0.01107241){\color[rgb]{0,0,0}\makebox(0,0)[lb]{\smash{$L(-1)$}}}%
    \put(0.94201182,0.01107243){\color[rgb]{0,0,0}\makebox(0,0)[lb]{\smash{$U(0)$}}}%
    \put(0,0){\includegraphics[width=\unitlength,page=2]{stevedore.pdf}}%
  \end{picture}%
\endgroup%